\documentclass[
  a4paper, 
  reqno, 
  oneside, 
  12pt
]{amsart}

\usepackage[utf8]{inputenc}

\usepackage[dvipsnames]{xcolor} 
\colorlet{cite}{red}
\usepackage{tikz}  
\usepackage{float}

\usetikzlibrary{arrows,positioning}
\tikzset{ 
  baseline=-2.3pt,
  text height=1.5ex, text depth=0.25ex,
  >=stealth,
  node distance=2cm,
  mid/.style={fill=white,inner sep=2.5pt},
}

\usepackage{lmodern}
\usepackage{tikz-cd}
\usepackage{amsthm, amssymb, amsfonts}

\usepackage{amsthm, amssymb, amsfonts}	
\usepackage[all]{xy}
\usepackage{graphicx,caption,subcaption}
\usepackage{braket}  
\usepackage{microtype}
\usepackage{DotArrow}
\usepackage[makeroom]{cancel}
\usepackage[%
  bookmarks=true,			
  unicode=true,			
  pdftitle={Lyapunov basic 1-forms},		%
  pdfauthor={Valencia},	%
  pdfkeywords={Lyapunov 1-form, flows, orbifolds},	
  colorlinks=true,		
  linkcolor=black,			
  citecolor=black,		
  filecolor=magenta,		
  urlcolor=RoyalBlue			
]{hyperref}				
\usepackage{cleveref}

\newtheoremstyle{mydef}
  {}		
  {}		
  {}		
  {}		
  {\scshape}	
  {. }		
  { }		
  {\thmname{#1}\thmnumber{ #2}\thmnote{ #3}}	

\newtheorem{theorem}{Theorem}[section]
\newtheorem*{theorem*}{Theorem}
\newtheorem{proposition}[theorem]{Proposition}
\newtheorem*{proposition*}{Proposition}
\newtheorem{lemma}[theorem]{Lemma}
\newtheorem*{lemma*}{Lemma}
\newtheorem{corollary}[theorem]{Corollary}
\newtheorem*{corollary*}{Corollary}
\theoremstyle{definition}
\newtheorem{definition}[theorem]{Definition}
\newtheorem{example}[theorem]{Example}

\theoremstyle{remark}
\newtheorem{remark}[theorem]{Remark}

\newtheorem*{conjecture*}{Conjecture}
\usepackage{tikz}

\newcommand{\rr}{\rightrightarrows}

\author{Fabricio Valencia}
\address{}
\date{\today}

\address{F. Valencia - Instituto de Matem\'atica e Estat\'istica, Universidade de S\~ao Paulo, Rua do Mat\~ao 1010, Cidade Universit\'aria, 05508-090 S\~ao Paulo - Brazil.
      \newline  
      \phantom{xx}
 fabricio.valencia@ime.usp.br}

\title{Lyapunov 1-forms on orbifolds}

\begin{document}
\maketitle

\begin{abstract}
We introduce and analyze a notion of smooth Lyapunov 1-form for flows generated by vector fields on orbifolds. Using asymptotic cycles and chain-recurrent sets, we establish topological conditions that guarantee the existence of a Lyapunov 1-form, lying in a prescribed cohomology class, for a given vector field on a compact orbifold.
\end{abstract}

\tableofcontents
\section{Introduction}

This paper is the third in a series devoted to the study of the geometric, topological, and dynamical properties of closed 1-forms on compact orbifolds. On the one hand, in \cite{V} the author initiated the study of closed 1-forms of Morse type on orbifolds, establishing Novikov-type inequalities that yield topological lower bounds for the number of zeros of a closed 1-form of Morse-type within a given cohomology class. On the other hand, in \cite{LopezGarciaValencia} the authors investigated the topological properties of the leaves of the singular foliation induced by a closed 1-form of Morse-type on a compact orbifold, and extended a celebrated result of Calabi, which provides a purely topological characterization of intrinsically closed harmonic 1-forms of Morse type. In line with these results, the aim of this paper is to introduce and characterize the notion of a Lyapunov 1-form for flows of vector fields on orbifolds, which can be achieved by building on several of the technical results developed in \cite{Hep2,PPT,V}. This work is motivated by the results concerning this notion in \cite{FKLZ1, FKLZ2, Lat}, where the manifold case was treated.

A continuous flow on a topological space $X$ is a continuous map $\Phi:\mathbb{R}\times X\to X$ such that $\Phi_0= \textnormal{id}_X$ and $\Phi_{\tau+\tau'}=\Phi_\tau\circ \Phi_{\tau'}$ for all $\tau,\tau'\in \mathbb{R}$, where $\Phi_\tau:X \to X$ is defined by $\Phi_\tau(x) = \Phi(\tau, x)$ for $\tau \in \mathbb{R}$ and $x\in X$. Accordingly, a \emph{Lyapunov function} for $\Phi$ is a continuous function $L:X\to \mathbb{R}$ which is non-increasing along its orbits, meaning that $L(\Phi_\tau(x))\leq L(x)$ for every $\tau\geq 0$ and $x\in X$. Such functions were introduced by Lyapunov in order to detect stable equilibria. More importantly, since Conley’s seminal works \cite{Conley1,Conley2}, Lyapunov functions have been used to study general dynamical systems, as they are strongly related to chain-recurrence properties. In the case where $X$ is a smooth manifold and $\Phi$ is the flow of a vector field, the problem of studying and characterizing smooth Lyapunov functions has been addressed by several authors. In particular, we refer to the works of Farber–Kappeler–Latschev–Zehnder \cite{FKLZ1,FKLZ2}, Latschev \cite{Lat}, and Fathi–Pageault \cite{FathiPageault}, which serve as an inspiration for the present paper. It is important to say that a more general notion that Lyapunov function, known as a Lyapunov 1-form, was introduced by Farber in \cite{Far}. Compared to Lyapunov functions, these objects allow one to go a step further and analyze the flow within the chain-recurrent set as well. Moreover, Lyapunov 1-forms provide an important tool for applying methods from homotopy theory to the study of dynamical systems. The smooth version of this notion was studied and initially developed in \cite{FKLZ1,FKLZ2,Lat}.

As discussed earlier, our goal is to introduce and analyze a notion of ``smooth Lyapunov 1-form'' for flows generated by vector fields on orbifolds. To this end, we model orbifolds using proper and étale Lie groupoids, which serve as atlases and allow us to regard an orbifold as the orbit space of such a groupoid. To state our main result, we first briefly introduce some terminology, following primarily the conventions and definitions in \cite{V}. Let $X$ be a compact orbifold, let $\overline{v}$ be a vector field on $X$ generating a continuous flow $\overline{\Phi}:\mathbb{R}\times X\to X$, and let $Y\subset X$ denote a closed subset which is flow-invariant.

\begin{definition}\label{Lyapunov1form}
	A closed 1-form $\overline{\omega}$ on $X$ is said to be a \emph{Lyapunov 1-form} for the pair $(\overline{\Phi},Y)$ if it has the following two properties:
	\begin{enumerate}
		\item the function $\iota_{\overline{v}}\overline{\omega}$ is negative on $X-Y$, and
		\item there exists a function $\overline{f}:\tilde{U}\to \mathbb{R}$ defined on an open neighborhood $\tilde{U}$ of $Y$ such that $\overline{\omega}|_{\tilde{U}}=d\overline{f}$ and $d\overline{f}|_Y=0$.
	\end{enumerate}
\end{definition}

In particular, if $\overline{\omega}$ is exact then the previous definition gives rise to a notion of smooth Lyapunov function on $X$ with respect to the flow $\overline{\Phi}$, which turns out to be constant along the components of $Y$.

Our primary objective is to find topological conditions which guarantee that for a given pair $(\overline{v},Y)$ on $X$ there exists a Lyapunov 1-form $\overline{\omega}$ lying in a prescribed cohomology class $\xi\in H^1(X)$. First of all, adapting the Schwartzman's construction for asymptotic cycles \cite{Sch} to the orbifold setting one can define a real homology class $\mathcal{A}_\mu(\Phi)\in H_1(X,\mathbb{R})$, associated to both the cohomology class $\xi$ and a fixed measure $\overline{\mu}$ on $X$ which is invariant under $\overline{\Phi}$. Second, the kernel of the homomorphism of periods of $\xi$ induces an orbifold covering space $\overline{p}_\xi: X_\xi\to X$, which can be used to lift the flow $\overline{\Phi}$ to a continuous flow $\tilde{\Phi}$ on $X_\xi$. Let $\overline{R}(\Phi)$ and $\overline{R}(\tilde{\Phi})$ denote the chain recurrent sets of the flows $\overline{\Phi}$ and $\tilde{\Phi}$, respectively. The chain recurrent set of $\xi$ is by definition $\overline{R}_\xi=\overline{p}_\xi(\overline{R}(\tilde{\Phi}))$, a closed and $\overline{\Phi}$-invariant set that sits inside $\overline{R}(\Phi)$. Its complement $\overline{R}(\Phi)-\overline{R}_\xi$ is denoted by $\overline{C}_\xi$. Third, let $\overline{d}$ be the geodesic distance induced by a fixed Riemannian metric on $X$. Given $\delta>0$ and $T>1$, a  $(\delta,T)$-cycle for the flow $\overline{\Phi}$ is defined as a pair $([x],\tau)$, where $[x]\in X$ and $\tau > T$, such that $\overline{d}(\overline{\Phi}_\tau([x]),[x])< \delta$. Assuming that $\delta$ is small enough, one can construct a unique homology class $z\in H_1(X,\mathbb{Z})$ for any $(\delta,T)$-cycle. As expected, all of the notions discussed above admit a description in terms of a proper étale Lie groupoid presenting $X$.

The main result of this paper can be stated as follows. 

\begin{theorem}\label{Main Thm}
	Let $\overline{v}$ be a vector field on a compact orbifold $X$ and let $\overline{\Phi}:\mathbb{R}\times X\to X$ denote the flow generated by $\overline{v}$. Let $\xi\in H^1(X)$ be a cohomology class such that the restriction $\xi|_{R_\xi}$, viewed as a $\Check{C}$ech cohomology class $\xi|_{R_\xi}\in \Check{H}^1(\overline{R}_\xi,\mathbb{R})$, vanishes. Then, the following properties of $\xi$ are equivalent:
	\begin{enumerate}
		\item[(a)] there exists a Lyapunov 1-form for the pair $(X,\overline{R}_\xi)$ in the cohomology class $\xi$ and the subset $\overline{C}_\xi$ is closed,
		\item[(b)] for any Riemannian metric on $X$ there exist $\delta>0$ and $T>1$ such that the homology class $z\in H_1(X,\mathbb{Z})$ associated with an arbitrary $(\delta,T)$-cycle $([x],\tau)$ of the flow $\overline{\Phi}$, with $[x]\in \overline{C}_{\xi}$, satisfies $\langle \xi,z\rangle \leq -1$,
		\item[(c)] the subset $\overline{C}_\xi$ is closed and there exists a constant $\lambda>0$ such that for any $\overline{\Phi}$-invariant positive measure $\overline{\mu}$ on $X$ the asymptotic cycle $\mathcal{A}_\mu\in H_1(X,\mathbb{R})$ verifies
		\begin{equation}\label{Condition 1}
			\langle \xi, \mathcal{A}_\mu\rangle \leq -\lambda \overline{\mu}(\overline{C}_\xi)
		\end{equation}
		and
		\item[(d)] the subset $\overline{C}_\xi$ is closed and for any $\overline{\Phi}$-invariant positive measure $\overline{\mu}$ on $X$  with $\overline{\mu}(\overline{C}_\xi)>0$, the asymptotic cycle $\mathcal{A}_\mu\in H_1(X,\mathbb{R})$ satisfies 
		\begin{equation}\label{Condition 2}
			\langle \xi, \mathcal{A}_\mu\rangle <0.
		\end{equation}
	\end{enumerate}
\end{theorem}

It is worth mentioning that, building on the works \cite{Hep2,PPT,V}, a similar result can be established for singular spaces more general than orbifolds, provided they can still be described as orbit spaces of proper Lie groupoids. Particularly interesting examples that fit our general assumptions are the singular orbit spaces arising as global quotients of compact Lie group actions.

The paper is structured as follows. In Section \ref{S:2}, we briefly introduce the main notions and terminology used throughout the paper. More precisely, we define basic vector fields and differential forms on proper étale Lie groupoids, which in turn model vector fields and differential forms on orbifolds. The notion of flows on orbifolds, and, more generally, on separated differentiable stacks, was studied in detail in \cite{Hep, Hep2}. We also define measures on orbifolds by following the elegant approach of \cite{CM}, which is based on transverse measures on Lie groupoids. Additionally, we introduce $G$-paths and define the homomorphism of periods of a given cohomology class, notions that were extensively studied in \cite{V}. Section \ref{S:3} is devoted to define and explore our notion of smooth Lyapunov 1-forms on orbifolds. We provide an alternative formulation of condition (2) in Definition \ref{Lyapunov1form}, and we show that the two formulations are equivalent if and only if either the chosen cohomology class is integral or the closed flow-invariant subset of the orbifold admits a neighborhood retraction, see Lemmas \ref{lema1} and \ref{Lema2}, respectively. We define asymptotic cycles and use them to determine a necessary condition for the existence of Lyapunov 1-forms on an orbifold. This is the content of Proposition \ref{Proposition1}. We also introduce chain recurrent sets for flows over orbifolds by using $(\delta,T)$-chains, $\delta>0$ and $T>1$. In particular, we define the chain-recurrent set associated with the cohomology class of a closed 1-form and establish the existence of a unique homology class for any closed $(\delta,T)$-chain, provided $\delta$ is sufficiently small. In Section \ref{S:4}, we prove Theorem \ref{Main Thm}. This is done by using the fact that our orbifolds can be regarded as compact, locally path-connected metric spaces, allowing us to apply the machinery developed in the previous sections to adapt, to the orbifold setting, the key arguments in the proof of Theorem 1 from \cite{FKLZ1}. Those arguments, in turn, rely on the analysis of continuous flows on compact, locally path-connected metric spaces carried out in \cite{FKLZ2}. Finally, motivated by the constructions developed in \cite[s. 7]{FKLZ2}, in Section \ref{S:5} we provide some examples which allow us to illustrate our main result.

\vspace{.2cm}
{\bf Acknowledgments:} Valencia was supported by Grant 2024/14883-6 Sao Paulo Research Foundation - FAPESP.

\section{Vector fields, differential forms, and measures}\label{S:2}

This section is devoted to briefly introduce the main notions and terminology that we will be using throughout the paper. We assume that the reader is familiar with the notion of Lie groupoid and the geometric/topological aspects underlying its structure \cite{dH,MoMr0}. Let $G\rr M$ be a Lie groupoid. The structural maps of $G$ are denoted by $(s,t,m,u,i)$ where $s,t:G\to M$ are the maps respectively indicating the source and target of the arrows, $m:G^{(2)}\to G$ stands for the partial composition of arrows, $u:M\to G$ is the unit map, and $i:G\to G$ is the map determined by the 
inversion of arrows, which sometime we shall denote by $u(x):=1_x$ for all $x\in M$. The orbit space $M/G$ of $G\rr M$ is denoted by $X$ and the canonical orbit projection by $\pi:M\to X$.

We say that a Lie groupoid $G\rr M$ is \emph{proper} if the source/target map $(s,t):G\to M\times M$ is proper. In this case, the groupoid orbits $\mathcal{O}_x$ are embedded in $M$, the isotropy groups $G_x$ are compact, and the orbit space $X$ is Hausdorff, second-countable, and paracompact \cite{dH}. Moreover, the orbit projection $\pi:M\to X$ is an open map. The Lie groupoid is said to be \emph{étale} if either $s$ or $t$ is a local diffeomorphism, meaning that $G$ and $M$ have the same dimension. We use proper étale groupoids as orbifold atlases, so that we can think of an orbifold as being the orbit space of a Lie groupoid of this kind \cite{BX,Ler}. Such a point of view turns out to be quite useful when dealing with several geometric and algebraic notions concerning the structure of an orbifold, or even a singular orbit space in general.

From now on, unless otherwise stated, we assume that $G\rr M$ is a proper étale Lie groupoid presenting an orbifold $X$.

\subsection{Basic vector fields and differential forms}

We can introduce our models for both vector fields and differentiable forms on $X$ by following the groupoid approach described for instance in \cite{HoffmanSjamaar,LerMal}. First, recall that given a surjective submersion $p:M\to M'$ of manifolds, a vector field $v\in \mathfrak{X}(M)$ is $p$-\emph{related} to a vector field $v'\in \mathfrak{X}(M')$ if $dp(v)=v'\circ p$. This is equivalent to saying that the local flows of $v$ and $v'$ commute with $p$. Note that if $v'$ is given then $v$ is determined up to a section of the bundle $\ker(dp)\subset TM$. Vector fields on $X$ are completely determined by basic vector fields on $M$. More precisely, the set of \emph{basic vector fields} on $M$, denoted by $\mathfrak{X}_{\textnormal{bas}}(G)$, is by definition the quotient
$$\mathfrak{X}_{\textnormal{bas}}(G)=\frac{\lbrace (v_1,v_0)\in \mathfrak{X}(G)\times \mathfrak{X}(M): ds(v_1)=v_0\circ s,\ dt(v_1)=v_0\circ t \rbrace}{\lbrace (v_1,v_0): ds(v_1)=v_0\circ s,\ dt(v_1)=v_0\circ t,\ v_1\in (\ker(ds)+\ker(dt)) \rbrace}.$$

That is, a basic vector field is not strictly a vector field, but a pair of equivalence classes of vector fields. Note that a basic vector field $\overline{v}= (\overline{v}_1, \overline{v}_0)$ is determined by its 2st component $\overline{v}_0$. It follows that $\mathfrak{X}_{\textnormal{bas}}(G)$ is a Lie algebra, and, more importantly, one can think of the Lie algebra of vector fields $\mathfrak{X}(X)$ as being equal to $\mathfrak{X}_{\textnormal{bas}}(G)$. This is consequence of the fact that the notion of basic vector field is Morita invariant \cite{HoffmanSjamaar,LerMal}.

\begin{remark}
	For future purposes, if $\overline{v}$ is a vector field on $X$ presented by a pair of $s,t$-related vector fields $(v_1,v_0)\in \mathfrak{X}(G)\times \mathfrak{X}(M)$ then we think of the local flow of $\overline{v}$ as being a 1-parametric family $\overline{\Phi}$ of local orbifold automorphisms of $X$ covered by the local flow of $v_0$ which, in turn, induces a 1-parametric family $\Phi$ of local automorphisms of $G$. That is to say, $\pi\circ \Phi =\overline{\Phi}\circ (\textnormal{id}\times \pi)$, where $\pi:M\to X$ stands for the orbit projection. This is motivated from the fact that the Lie algebra of basic vector fields is quasi-isomorphic to the Lie 2-algebra of multiplicative vector fields, consult \cite[s. 5.4]{HoffmanSjamaar}. Besides, particular instances of this situation are provided by orbifolds arising as global quotients of compact Lie group actions \cite[s. 6]{Hep2}.
	
	The notions of vector fields and flows for orbifolds, and more generally, for separated differentiable stacks, were studied and explored in detail by Hepworth in \cite{Hep,Hep2}. 
\end{remark}

Second, it is well-known that differential forms on $X$ are completely determined by basic forms on $M$, compare also \cite{PPT,TuX,Wa}. That is, differential forms $\omega$ on $M$ such that $(t^\ast-s^\ast)(\omega)=0$. The space of basic differential forms on $M$ is denoted by $\Omega_{\textnormal{bas}}^\bullet(G)$ and it is identified with the space of differential forms on $X$ that, in turn, is denoted by $\Omega^\bullet(X)$. Note that the de Rham differential over $\Omega^\bullet(M)$ restricts to $\Omega_{\textnormal{bas}}^\bullet(G)$, thus yielding the so-called \emph{basic cohomology} $H_{\textnormal{bas}}^\bullet(G,\mathbb{R})$ of $G$. Such a cohomology is Morita invariant, so that its recover the de Rham cohomology $H^\bullet(X)$ of $X$. Additionally, there is an isomorphism $H^\bullet(X,\mathbb{R})\cong H_{\textnormal{bas}}^\bullet(G,\mathbb{R})$, where $H^\bullet(X,\mathbb{R})$ stands for the singular cohomology of $X$. It is worth mentioning that $X$ can be triangulated so that its cohomology $H^\bullet(X)\approx H_{\textnormal{bas}}^\bullet(G,\mathbb{R})$ becomes a finite dimensional vector space \cite{PPT}. It is important to have in mind that one also has a groupoid version of the Poincaré Lemma, saying that for each groupoid orbit $\mathcal{O}$ there exist an open neighborhood $\mathcal{O}\subset U\subset M$ and a basic smooth function $f_U\in \Omega_{\textnormal{bas}}^0(G|_{U})$ such that $\omega|_U=df_U$, see \cite[Lem. 8.5]{PPT}.

If we identify the space of basic forms $\Omega_{\textnormal{bas}}^\bullet(G)$ with the set of pairs $\lbrace (\theta_1,\theta_0)\in \Omega^\bullet(G)\times \Omega^\bullet(M): s^\ast(\omega_0)=\omega_1= t^\ast(\omega_0)\rbrace$ then we have contraction operations and Lie derivatives $\iota: \mathfrak{X}_{\textnormal{bas}}(G)\times \Omega_{\textnormal{bas}}^\bullet(G)\to \Omega_{\textnormal{bas}}^{\bullet-1}(G)$ and $\mathcal{L}:\mathfrak{X}_{\textnormal{bas}}(G)\times \Omega_{\textnormal{bas}}^\bullet(G)\to \Omega_{\textnormal{bas}}^{\bullet}(G)$ respectively defined by

$$\iota_{\overline{v}}\omega=(\iota_{v_1}\omega_1,\iota_{v_0}\omega_0)\quad\textnormal{and}\quad \mathcal{L}_{\overline{v}}\omega=(\mathcal{L}_{v_1}\omega_1,\mathcal{L}_{v_0}\omega_0),$$
where $(v_1,v_0)\in \mathfrak{X}(G)\times \mathfrak{X}(M)$ is a representative of $\overline{v}$. As expected, these expressions do not depend on the choice of $(v_1,v_0)$, see \cite{HoffmanSjamaar}.

\subsection{Transverse measures}

Measures over orbifolds, and more generally, over separated differentiable stacks were recently studied and explored in detail by Crainic--Mestre in \cite{CM}. We begin this subsection by considering a proper Lie groupoid $G\rr M$ presenting an orbit space $X$, and later specialize to the particular case of orbifolds. Recall that one can think of the algebra of smooth functions on $X$ as
$$C^\infty(X)=\lbrace \overline{f}:X\to \mathbb{R}: f\circ \pi \in C^\infty(M)\rbrace\cong C^\infty(M)^G,$$
where $C^\infty(M)^G=\lbrace f\in C^\infty(M): f\circ s=f\circ t \rbrace$ stands for the algebra of smooth basic functions on $M$. The algebra of compactly supported smooth functions is defined as usual, by just adding the compact support condition:
$$C^\infty_c(X)=\lbrace \overline{f}\in C^\infty(X): \textnormal{supp}(\overline{f})-\textnormal{compact} \rbrace.$$

Of course, one can define similarly $C^k_c(X)$ for any integer $k\geq 0$. For $k=0$, since $X$ is endowed with the quotient topology, one recovers the usual space of compactly supported continuous functions on the space $X$. 

If $\mu$ is a measure on $X$ then we denote by $I_\mu:C_c(X)\to \mathbb{R}$ its corresponding integration map which is defined by $I_\mu(\overline{f})=\int_X\overline{f}([x])d\mu([x])$. As in the case of smooth manifolds, one has that (see \cite[Prop. 5.2]{CM}):
\begin{itemize}
\item $C^\infty_c(X)$ is dense in $C_c(X)$,
\item any positive linear functional on $C_c(X)$ is automatically continuous, with the same holding for $C^\infty_c(X)$ with the induced topology, and
\item the construction $(\mu\mapsto \mu|_{C^\infty_c(X)})\cong (I\mapsto I|_{C^\infty_c(X)})$ induces a one-to-one correspondence between measures on $X$ and positive linear functionals on $C^\infty_c(X)$.
\end{itemize}

More importantly, there exists a one-to-one correspondence between transverse measures on $G\rr M$ and measures on $X$. In order to be precise, we need to introduce some terminology. The density bundle of any vector bundle $E$ is denoted by $\mathcal{D}_E$, and its sections are called \emph{densities} on $E$. We denote by $\mathcal{D}(M)$ the space of densities on $TM$. Recall that the measure associated with a positive density $\rho\in \mathcal{D}(M)$ is by definition
$$\mu_\rho: C^\infty_c(M)\to \mathbb{R},\qquad \mu_\rho(f)=\int_M f\cdot \rho,$$
where $\int_M:C^\infty_c(M,\mathcal{D}_{TM})\to \mathbb{R}$ is the canonical integration on $M$ (consult e.g. \cite[p. 1243]{CM}). Measures of this type are called \emph{geometric measures}. 

Let us consider a surjective submersion $p:M\to M'$. We define the \emph{push-forward} $p_!(I)(f):=I(f\circ p)$, or, in the integral notation
$$\int_{M'}f(x')d\mu_{p_!(I)}(x'):=\int_M f(p(x)) d\mu(x),$$
for any $x\in p^{-1}(x')$. It follows that the construction $I\mapsto p_!(I)$ takes geometric measures on $M$ to geometric measures on $M'$. This is because integration over the fibers makes sense for densities, so that $p_!=\int_{\textnormal{fiber}}:\mathcal{D}(M)\to \mathcal{D}(M')$, where $p_!(\rho)(x'):=\int_{p^{-1}(x')}\rho\in \mathcal{D}_{T_{x'}M'}$ since we can interpret $\rho|_{p^{-1}(x')} \in \mathcal{D}(p^{-1}(x'))\otimes \mathcal{D}_{T_{x'}M'}$. Additionally, by the Fubini formula for densities it holds that this operation is compatible with the one on measures, i.e. that $p_!(I_\rho)=I_{p_!(\rho)}$.

Let $A\to M$ denote the Lie algebroid of $G$. Following \cite[s. 3.3]{CM}, one has canonical integrations over the $s,t$-fiber maps
$$s_!,t_!:C^\infty_c(G,t^\ast \mathcal{D}_A\otimes s^\ast \mathcal{D}_A)\to C^\infty_c(M,\mathcal{D}_A).$$

A \emph{transverse measure} on $G\rr M$ is a positive linear map $\mu:C^\infty_c(M,\mathcal{D}_A)\to \mathbb{R}$ satisfying the invariance condition $\mu\circ s_!=\mu\circ t_!$. These kinds of measures always exist and can be constructed by using averaging techniques, compare \cite[Prop. 5.1]{CM}. Furthermore, there is an intrinsic averaging map
$$\textnormal{Av}: C^\infty_c(M,\mathcal{D}_A)\to C^\infty_c(X),\qquad \textnormal{Av}(\rho)(\pi(x))=\int_{s^{-1}(x)}\rho|_{s^{-1}(x)},$$
which restrict to an isomorphism between $C^\infty_c(M,\mathcal{D}_A)/\textnormal{im}(s_!-t_!)$ and $C^\infty_c(X)$, thus yielding a one-to-one correspondence between transverse measures on $G\rr M$ and measures on $X$.

In the particular case that $X$ is an orbifold, we obtain that $A$ is trivial, as $G\rr M$ is étale, so that the above correspondence reads as follows. There exists a one-to-one correspondence between measures $\overline{\mu}$ on $X$ and $G$-invariant measures $\mu$ on $M$ which is uniquely determined by $\overline{\mu}=\mu\circ \pi_!$ where
$$\pi_!:C^\infty_c(M)\to C^\infty_c(X),\qquad \pi_!(f)(\pi(x))=\sum_{g\in s^{-1}(x)}f(t(g)).$$ 

One can write down a relation between $\mu$ and $\overline{\mu}$ by using a \emph{cut-off function}, i.e. a smooth positive function $c:M\to \mathbb{R}$ such that:
\begin{itemize}
\item the restriction of $s$ to $t^{-1}(\textnormal{supp}(c))$ is proper (as a map to $M$), and 
\item $\int_{s^{-1}(x)}c(t(g))d\mu^x(d)=1$ for all $x\in M$, where $\lbrace \mu^x\rbrace_{x\in M}$ is a proper Haar measure system for $G$.
\end{itemize}

Cut-off functions always exist over proper Lie groupoids. In particular, we have the following relation
$$\int_X \overline{f}([x])d\overline{\mu}([x])=\int_M c(x)f(x)d\mu(x).$$

As expected, such a relation does not depend on the choice of the cut-off function. See \cite{CM} for concrete details.

\subsection{$G$-path integrals}

Let us introduce the notion of $G$-path in $M$ which is a crucial concept in this paper. The reader is recommended to consult \cite[s. 3.3]{MoMr} and \cite[c. G; s. 3]{BH} for specific details. We think of paths in $X$ as $G$-paths in $M$. Namely, a \emph{$G$-path} in $M$ is a sequence $\sigma:=\sigma_ng_n\sigma_{n-1}\cdots \sigma_1g_1\sigma_0$ where $\sigma_0,\cdots,\sigma_n:[0,1]\to M$ are (piecewise smooth) paths in $M$ and $g_1,\cdots,g_n$ are arrows in $G$ such that $g_j:\sigma_{j-1}(1)\to \sigma_{j}(0)$ for all $j=1,\cdots,n$. One says that $\sigma$ is a $G$-path of \emph{order} $n$ from $\sigma_{0}(0)$ to $\sigma_{n}(1)$. Our groupoid $G$ is said to be \emph{$G$-connected} if for any two points $x,y\in M$ there exists a $G$-path from $x$ to $y$. We always assume that the groupoids we are working with are $G$-connected unless otherwise stated. If $\sigma':=\sigma_n'g_n'\sigma_{n-1}'\cdots \sigma_1'g_1'\sigma_0'$ is another $G$-path in $M$ with $\sigma_{0}'(0)=\sigma_{n}(1)$ then we can \emph{concatenate} $\sigma$ and $\sigma'$ into a new $G$-path 
$$\sigma\ast \sigma'=\sigma_n'g_n'\sigma_{n-1}'\cdots \sigma_1'g_1'\sigma_0'1_{\sigma_{n}(1)}\sigma_ng_n\sigma_{n-1}\cdots \sigma_1g_1\sigma_0.$$

We define an equivalence relation in the set of $G$-paths in $M$ which is generated by the following \emph{multiplication equivalence}
$$\sigma_n g_n\cdots \sigma_{j+1}g_{j+1}\sigma_jg_j\sigma_{j-1}\cdots g_1\sigma_0 \quad\sim\quad \sigma_n g_n\cdots \sigma_{j+1}g_{j+1}g_j\sigma_{j-1}\cdots g_1\sigma_0,$$
if $\sigma_j$ is the constant path for any $0<j<n$, and  \emph{concatenation equivalence}
$$\sigma_n g_n\cdots g_{j+1}\sigma_jg_j\sigma_{j-1}g_{j-1}\cdots g_1\sigma_0 \quad\sim\quad \sigma_n g_n\cdots g_{j+1}\sigma_j\cdot\sigma_{j-1}g_{j-1}\cdots g_1\alpha_0,$$
if $g_j=1_{\alpha_{j-1}(1)}$ for any $0<j<n$ where $\sigma_j\cdot\sigma_{j-1}$ stands for the standard concatenation of the paths $\sigma_j$ and $\sigma_{j-1}$. A \emph{deformation} between two $G$-paths $\sigma$ and $\sigma'$ of the same order $n$ from $x$ to $y$ consist of homotopies $D_j:[0,1]\times [0,1]\to M$ from $\sigma_j$ to $\sigma_j'$ for $j=0,1,\cdots, n$ and paths $d_i:[0,1]\to G$ from $g_j$ to $g_j'$ for $j=1,\cdots, n$ such that $s\circ d_j=D_{j-1}(\cdot,1)$ and $t\circ d_j=D_{j}(\cdot,0)$ for $j=1,\cdots, n$ verifying $D_0([0,1],0)=x$ and $D_n([0,1],1)=y$. That is, a deformation as a continuous family of $G$-paths of order $n$ from $x$ to $y$ which may be written as $D_n(\tau,\cdot)d_n(\tau)\cdots d_1(\tau)D_0(\tau,\cdot)$ for $\tau\in [0,1]$. Accordingly, two $G$-paths with fixed endpoints in $M$ are said to be \emph{$G$-homotopic} if it is possible to pass from one to another by a sequence of equivalences and deformations.

With the multiplication induced by concatenation it follows that the $G$-homotopy classes of $G$-paths form a Lie groupoid over $M$ which is called the \emph{fundamental groupoid} of $G$ and is denoted by $\Pi_1(G)\rr M$. The \emph{fundamental group} of $G$ with respect to a base-point $x_0\in M$ is the isotropy group $\Pi_1(G, x_0):= \Pi_1(G)_{x_0}$. Note that it consists of $G$-homotopy classes of $G$-loops at $x_0$ which are by definition the $G$-homotopy classes of $G$-paths from $x_0$ to $x_0$. It is simple to check that for any two different points $x_0,y_0\in M$ it holds that $\Pi_1(G, x_0)$ and $\Pi_1(G, y_0)$ are isomorphic by $G$-connectedness. Every Lie groupoid morphism $\phi_1:G\to G'$ covering $\phi_0:M\to M'$ induces another Lie groupoid morphism $(\phi_1)_\ast:\Pi_1(G)\to \Pi_1(G')$ by mapping $[\sigma]$ to $[(\phi_1)_\ast(\sigma)]$, where $(\phi_1)_\ast(\sigma)=(\phi_0)_\ast(\sigma_n)\phi_1(g_n)(\phi_0)_\ast(\sigma_{n-1})\cdots (\phi_0)_\ast(\sigma_1)\phi_1(g_0)(\phi_0)_\ast(\sigma_0)$. This also covers $\phi_0:M\to M'$ so that it induces a Lie group homomorphism between the fundamental groups  $(\phi_1)_\ast:\Pi_1(G,x_0)\to \Pi_1(G',\phi_0(x_0))$. As an important feature, we have that Morita equivalent groupoids have isomorphic fundamental groupoids, so that we define the \emph{orbifold fundamental group} of $X$ at $[x_0]$ as $\Pi_1^\textnormal{orb}(X,[x_0])=\Pi_1(G,x_0)$.

Let $\overline{\omega}$ be a closed 1-form on $X$ presented by a closed basic 1-form $\omega$ on $M$. For each smooth $G$-path $\sigma=\sigma_ng_n\sigma_{n-1}\cdots \sigma_1g_1\sigma_0$ in $M$ we define the $G$-\emph{path integral}
$$\int_{\sigma}\overline{\omega}=\sum_{k=0}^{n}\int_{\sigma_k}\omega.$$

If $[\sigma]$ denotes the $G$-homotopy class of $\sigma$ then the expression $\int_{[\sigma]}\overline{\omega}=\int_{\sigma}\overline{\omega}$ is well-defined. More importantly, suppose that $\omega$ and $\omega'$ are cohomologous closed basic 1-forms. That is, there is a basic smooth function $f:M\to \mathbb{R}$ such that $\omega-\omega'=df$. Hence, it is simple to check that $\int_\sigma (\omega-\omega')=f(\sigma_n(1))-f(\sigma_0(0))$, implying that the expression $\int_{[\sigma]}\xi=\int_{\sigma}\overline{\omega}$ is also well-defined for the cohomology class $\xi=[\omega]\in H^\bullet(X)$, provided that $\sigma$ is a $G$-loop. Additionally, we get a well-defined group homomorphism $\textnormal{Per}_{\xi}:\Pi_1^\textnormal{orb}(X,[x_0])\to (\mathbb{R},+)$ by sending $[\sigma]\mapsto \int_{\sigma}\overline{\omega}$, which is completely determined $\xi$. We refer to $\textnormal{Per}_{\xi}$ as the \emph{homomorphism of periods} associated with the cohomology class $\xi$. See \cite{V} for details.

\section{Lyapunov basic 1-forms}\label{S:3}

Let $X$ be an orbifold presented by a proper étale Lie groupoid $G\rr M$. From now on, we fix a vector field $\overline{v}$ on $X$ generating a continuous flow $\overline{\Phi}:\mathbb{R}\times X\to X$ and a closed subset $Y\subset X$ which is flow-invariant \cite{Hep,Hep2}. We assume that $\overline{v}$ is presented by a pair of vector fields $(v_1,v_0)\in \mathfrak{X}(G)\times \mathfrak{X}(M)$ where $v_1$ is $s,t$-related to $v_0$ such that:
\begin{itemize}
\item $v_0$ generates a continuous flow $\Phi:\mathbb{R}\times M\to M$ covering $\overline{\Phi}$, i.e. $\pi\circ \Phi =\overline{\Phi}\circ (\textnormal{id}\times \pi)$, where $\pi:M\to X$ stands for the orbit projection, and
\item if $H\rr N$ is a topological subgroupoid of $G\rr M$, with $N\subset M$ closed, presenting $Y$ then $N$ is flow-invariant with respect to the flow of $v_0$. We may think of $N$ as being $\pi^{-1}(Y)$.
\end{itemize}

Motivated by the notion of smooth Lyapunov 1-form introduced in \cite{FKLZ1,FKLZ2} for flows on manifolds, we set up the following definition.

\begin{definition}\label{MainDefinition}
A closed 1-form $\overline{\omega}$ on $X$ is said to be a \emph{Lyapunov 1-form} for the pair $(\overline{\Phi},Y)$ if it has the following two properties:
\begin{enumerate}
\item the function $\iota_{\overline{v}}\overline{\omega}$ is negative on $X-Y$, and
\item there exists a function $\overline{f}:\tilde{U}\to \mathbb{R}$ defined on an open neighborhood $\tilde{U}$ of $Y$ such that $\overline{\omega}|_{\tilde{U}}=d\overline{f}$ and $d\overline{f}|_Y=0$.
\end{enumerate}
\end{definition}

Let $\omega$ be a closed basic 1-form on $M$ presenting $\overline{\omega}$. Thus, a couple of observations come in order. First, recall that the function in (1) is given by $\iota_{\overline{v}}\overline{\omega}=\iota_{v_0}\omega$, expression that is well-defined since $\omega$ is basic. Therefore, condition (1) amounts to asking that the function $\iota_{v_0}\omega$ is negative on $M-N$. Second, if $U=\pi^{-1}(\tilde{U})$ is the corresponding neighborhood of $N$ in $M$ then there is a smooth basic function $f:U\to\mathbb{R}$ presenting $\overline{f}$ such that $\omega|_{U}=df$ and $df|_N=0$. In fact, the latter requirement is actually equivalent to condition (2).

\begin{remark}
If $\omega$ is additionally exact, i.e. there is a smooth basic function $f:M\to \mathbb{R}$ such that $\omega= df$, then Definition \ref{MainDefinition} gives rise to a notion of Lyapunov function for orbifolds. In fact, condition (1) implies that, for any $x\in M-N$ and $\tau>0$, it holds that $\int_{\gamma_x} df =f(\Phi_\tau(x))-f(x)<0$, thus obtaining $f(\Phi_\tau(x))<f(x)$, where $\gamma_x(\tau)$ denotes the integral curve of $v_0$ starting at $x$. Additionally, condition (2) implies that $f$ is constant along the connected components of $N$. Therefore, the induced function $\overline{f}:X\to \mathbb{R}$ enjoys the same properties over $Y$, so that we might think of it as being a Lyapunov function for $(\overline{\Phi},Y)$.
\end{remark} 

As in the manifold case, there are several alternatives for condition (2), for instance:
\begin{itemize}
\item[(2')] The 1-form $\omega$, viewed as a smooth section $\omega:M\to T^\ast M$, vanishes on $N$.
\end{itemize}

Of course, (2) implies (2'). The converse is also true in some special cases, as next results show.

The basic cohomology class $\xi$ of $\omega$ is said to be \emph{integral}, i.e. $\xi\in H^1_{\textnormal{bas}}(G,\mathbb{Z})$, if its homomorphism of periods $\textnormal{Per}_{\xi}$ takes values in $(\mathbb{Z},+)$, see \cite{V}.

\begin{lemma}\label{lema1}
If the basic cohomology class $\xi$ of $\omega$ is integral then conditions (2) and (2') are equivalent.
\end{lemma}

\begin{proof}
Let us suppose that $\xi\in H^1_{\textnormal{bas}}(G,\mathbb{Z})$. By \cite[Prop. 3.4]{V} it follows that there exists a basic smooth function $h:M\to S^1$ such that $\omega=ah^{\ast}(d\theta)$, where $a\in \mathbb{R}$ is a constant and $d\theta$ is the angle form on $S^1$. Let $\beta$ be a regular value of $h$, so that $(s^\ast h)^{-1}(\beta)\rr h^{-1}(\beta)$ is another proper étale Lie subgroupoid of $G\rr M$. Condition (2') implies that $U=M-h^{-1}(\beta)$ is an open neighborhood of $N$. Additionally, once again by \cite[Prop. 3.4]{V}, for a fixed base point $x_0\in M$ we know that $h$ can be written over $U$ as $h(x)=\exp(2\pi\sqrt{-1}\int_{\sigma_x}\omega)$ where $\sigma_x$ is any $G|_U$-path in $U$ joining $x_0$ with $x$. It is simple to check that the smooth function $f:U\to \mathbb{R}$ defined by sending $x\mapsto \int_{\sigma_x}\omega$ is basic and it verifies that $\omega|_U=df$. Hence, conditions (2) follows as desired.
\end{proof}

We say that $N$ admits a \emph{groupoid neighborhood retraction} if there exist an open groupoid neighborhood $(G|_U\rr U)\subset (G\rr M)$ of $H\rr N$ and a groupoid retraction\footnote{That is, $r$ is both a groupoid morphism and a retraction.} $r:G|_U\to H$ such that the inclusion $\iota_U:G|_U \to G$ is groupoid homotopic\footnote{The homotopy preserves the groupoid structure in the sense that it is given by a 1-parametric family of groupoid morphisms.} to $\iota_N\circ r$, where $\iota_N:H\to G$ is the inclusion. Note that the pair $(U,r)$ induces a corresponding orbifold neighborhood retraction $(\pi(U),\overline{r})$ for $Y$ in $X$.

\begin{lemma}\label{Lema2}
If $N$ admits a groupoid neighborhood retraction then conditions (2) and (2') are equivalent.
\end{lemma}

\begin{proof}
Let us fix a base point $x_j$ in each $G$-path component $U_j$ of $U$ above. One can define basic smooth functions $f_j:U_j\to \mathbb{R}$ as $f_j(x)=\int_{\sigma_x}\omega$ where $\sigma_x$ is any $G|_{U_j}$-path in $U_j$ joining $x_j$ with $x$. We need to verify that each $f_j$ is well-defined by showing that the latter $G$-path integral is independent of the choice of the integration $G$-path $\sigma_x$. Note that this amounts to asking that the $G$-path integral $\int_\sigma \omega$ vanishes for any $G|_U$-loop $\sigma=\sigma_ng_n\sigma_{n-1}\cdots \sigma_1g_1\sigma_0$ lying in $U$. Let us consider the $G$-loop $\tilde{\sigma}=r\circ \sigma$ in $N$ which is given by
$$\tilde{\sigma}=(r_0\circ \sigma_n)r_1(g_n)(r_0\circ \sigma_{n-1})\cdots (r_0\circ \sigma_1)r_1(g_0)(r_0\circ \sigma_0).$$

Let $D:[0,1]\times G|_U\to G$ denote the groupoid homotopy between $\iota_U$ and $\iota_N\circ r$. For each $k=0,1,\cdots, n$ we have that the path $\sigma_k$ is homotopic to $r_0\circ \sigma_k$ with homotopy $D_k:[0,1]\times [0,1]\to M$ defined by $D_k(\tau,\tilde{\tau})=D(\tau,\sigma_k(\tilde{\tau}))$. We also set $d_k:[0,1]\to G$ as $d_k(\tau)=D(\tau,g_k)$ for each arrow $g_k\in G|_U$ arising in the $G|_U$-loop $\sigma$. Hence, $D_n(\tau,\cdot)d_n(\tau)\cdots d_1(\tau)D_0(\tau,\cdot)$ defines a $G$-homotopy between $\sigma$ and $\tilde{\sigma}$, as $r:G|_U\to H$ is a groupoid retraction.

By \cite[Lem. 3.1]{V} together with condition (2') it follows that $\int_{\sigma}\omega=\int_{\tilde{\sigma}}\omega=0$. Finally, the basic functions $f_j$ determine together a basic smooth function $f:U\to \mathbb{R}$ with $\omega|_U=df$.
\end{proof}

It is shown in \cite[Lem. 3.6]{V} that the space of zeros of $\omega$ is saturated, so that we have a topological subgroupoid $\textnormal{zeros}(s^\ast \omega)\rr \textnormal{zeros}(\omega)$ of $G\rr M$. Thus, as an immediate consequence of the previous results we get that:
\begin{corollary}
Suppose that $N=\textnormal{zeros}(\omega)$. If either the basic cohomology class $\xi$ of $\omega$ is integral or else $N$ is admits a groupoid neighborhood retraction then conditions (2) is redundant.
\end{corollary}

We want to find topological conditions which guarantee that for a given pair $(\overline{v},Y)$ on $X$ there exists a Lyapunov 1-form $\overline{\omega}$ lying in a prescribed cohomology class $\xi\in H^1(X)$.

\subsection{Asymptotic cycles}

It is well-known that the singular real homology groups $H_\bullet(X,\mathbb{R})$ of $X$ can be described in terms of the simplicial structure of the nerve of $G\rr M$, more precisely, as their associated total real homology groups \cite{Be}. We can easily adapt Schwartzman's construction for asymptotic cycles \cite{Sch} to our setting as follows. Assume that $X$ is compact. Let $\overline{\mu}$ be a measure on $X$ which is invariant under $\overline{\Phi}$. It is clear that such a measure corresponds to a transverse measure $\mu$ on $M$ which is invariant under $\Phi$. Out of this data one can define a real homology class $\mathcal{A}_\mu(\Phi)\in H_1(X,\mathbb{R})$. Indeed, for a basic cohomology class $\xi\in H^1_{\textnormal{bas}}(G,\mathbb{R})=H^1(X)$ we define the \emph{evaluation map} given by the integral
\begin{equation}\label{AsymCycle}
\langle \xi, \mathcal{A}_\mu \rangle=\int_{X}\iota_{\overline{v}}\overline{\omega}([x]) d\overline{\mu}([x])=\int_{M}c(x)\iota_{v_0}\omega(x) d\mu(x),
\end{equation}
where $\omega$ is the closed basic 1-form on $M$ presenting $\overline{\omega}$, and the last equality is well-posed, provided $c:M\to \mathbb{R}$ is cut-off function, see \cite[Prop. 5.4]{CM}. In order to see that the expression above does not depend on the choice of $\omega$ we replace $\omega$ by $\omega'=\omega+df$ where $f:M\to\mathbb{R}$ is a smooth basic function, so that the integral \eqref{AsymCycle} gets changed by the equality 
$$\int_{M} c(x)(v_0)_x(f)d\mu(x)=\lim_{\tau \mapsto 0}\frac{1}{\tau}\int_{M}c(x)(f(\Phi_\tau(x))-f(x))d\mu(x).$$

Since $\mu$ is flow invariant we obtain that the right-hand side of the latter expression vanishes of any $f\in C^\infty(M)^G$, as $v_0$ covers a basic vector field. This automatically implies that $\langle \xi, \mathcal{A}_\mu \rangle$ is well-defined. Additionally, it is simple to see that the right-hand side of \eqref{AsymCycle} is a linear function on $\xi\in H^1_{\textnormal{bas}}(G,\mathbb{R})\cong H^1(X,\mathbb{R})$. Therefore, there exists a unique real homology class  $\mathcal{A}_\mu(\Phi)\in H_1(X,\mathbb{R})$ which satisfies    \eqref{AsymCycle} for all $\xi\in H^1_{\textnormal{bas}}(G,\mathbb{R})$. This is called the \emph{asymptotic cycle} of the flow $\overline{\Phi}$ corresponding to the measure $\overline{\mu}$.

Let us fix $\overline{\Phi}$ and vary the invariant measure $\overline{\mu}$. The fact that $\mathcal{A}_\mu$ depends linearly on $\overline{\mu}$ ensures that the set of asymptotic cycles $\mathcal{A}_\mu$ corresponding to all $\overline{\Phi}$-invariant positive measures on $X$ determines a cone in the vector space $H_1(X,\mathbb{R})$. Similarly to the manifold case \cite[Prop. 1]{FKLZ1}, a necessary condition for the existence of Lyapunov 1-forms on $X$ is given as follows.

\begin{proposition}\label{Proposition1}
Assume that there exits a Lyapunov 1-form on $X$ for $(\overline{v},Y)$, lying in a cohomology class $\xi\in H^1(X)$. Then $\langle \xi, \mathcal{A}_\mu \rangle\leq 0$ for any $\overline{\Phi}$-invariant positive measure $\overline{\mu}$ on $X$, with the equality holding if and only if the complement of $Y$ has measure zero. Moreover, the restriction of $\xi$ to $Y$, viewed as a $\Check{C}$ech cohomology class $\xi|_{Y}\in \Check{H}^1(Y,\mathbb{R})$, vanishes.
\end{proposition}
\begin{proof}
By Definition \ref{MainDefinition} we get that the basic smooth function $\iota_{v_0}\omega$ is negative on $M-N$ and vanishes on $N$ (resp. the function $\iota_{\overline{v}}\overline{\omega}$ is negative on $X-Y$ and vanishes on $Y$). This implies that the integral $\int_{M}c(x)\iota_{v_0}\omega(x) d\mu(x)$ has to be non-positive, as $c:M\to \mathbb{R}$ is always positive. By the Riesz's theorem \cite[p. 256]{Lang} we get that if $\overline{\mu}(X-Y)>0$ then there exists a compact $K\subset X-Y$ with $\overline{\mu}	(K)>0$.  Hence, one can choose $\epsilon>0$ such that $\iota_{\overline{v}}\overline{\omega}|_{K}=\iota_{v_0}\omega|_{\pi^{-1}(K)}\leq -\epsilon$, so that
$$\langle \xi, \mathcal{A}_\mu \rangle=\int_{M}c(x)\iota_{v_0}\omega(x) d\mu(x)<-\epsilon \mu(\pi^{-1}(K))<0.$$

In other words, we have shown that the value of $\langle \xi, \mathcal{A}_\mu \rangle$ is strictly negative if the measure $\overline{\mu}$ is not supported in $Y$. 

As expected, the $\Check{C}$ech cohomology $\Check{H}^\bullet(Y,\mathbb{R})$ can be constructed as the total $\Check{C}$ech cohomology which is defined out of the simplicial structure of the nerve of $H\rr N$, consult \cite{Be} and \cite[s. 3.3]{BX}. It is well-known that $\Check{H}^1(Y,\mathbb{R})$ can be obtained as the direct limit of the singular cohomology $$\Check{H}^\bullet(Y,\mathbb{R})=\lim_{W\supset Y,\ W\ \textnormal{open}}H^1(W,\mathbb{R}),$$
and that such a result is compatible with the simplicial constructions defining both $\Check{H}^\bullet (Y,\mathbb{R})$ and $H^\bullet (W,\mathbb{R})$. Thus, the last assertion of our main statement follows directly, as $0=\xi|_{U}\in H^1(U)$ by condition (2) in Definition \ref{MainDefinition}.
\end{proof}

\subsection{Chain-recurrent sets}

A Riemannian metric on $X$ can be thought of as an equivalence class of a groupoid Riemannian metric on $G\rr M$ in the sense of del Hoyo and Fernandes \cite{dHF,dHF2}. The manifolds of arrows $G$ and objects $M$ respectively inherit Riemannian metrics $\eta^{(1)}$ and $\eta^{(0)}$ such that $s$ and $t$ become local isometries and the inversion $i$ gives rise to an isometry. Additionally, the induced bundle metrics along the normal directions to the orbits are such that the normal isotropy representations are by linear isometries. This provides us with a way of inducing ``inner products'' over the ``coarse'' tangent spaces $T_{[x]} X \approx \nu_x(\mathcal{O}_x)/G_x$ for all $x\in M$.

Following \cite[Thm. 6.1]{PPT}, we have that the geodesic distance $\overline{d}$ on $X$ can be obtained as
$$\overline{d}([x],[y])=\inf\lbrace d(x_1,\mathcal{O}_x)+\cdots+d(x_k,\mathcal{O}_{x_{k-1}})\rbrace,$$
where $d$ stands for the geodesic distance associated with $\eta^{(0)}$ on $M$, and the infimum is taken over all choices of pairs $(x_j,\mathcal{O}_{x_j})$ for $j=1,\cdots,k-1$ and $x_k\in \mathcal{O}_y$, over all $k\in \mathbb{N}$. Additionally, the distance $\overline{d}$ has the following properties:
\begin{itemize}
\item it is uniquely determined by the property that for each orbit $\mathcal{O}$ in $M$ and every point $z\in T_{\mathcal{O}}$ of an appropriate metric tubular neighborhood $T_{\mathcal{O}}$ of $\mathcal{O}$ the relation $\overline{d}(\mathcal{O},\mathcal{O}_z)=d(z,\mathcal{O})$ holds true, where $\mathcal{O}_z$ is the orbit through $z$,
\item the canonical projection $\pi:M\to X$ is a submetry, and
\item the topology induced on $X$ by $\overline{d}$ coincides with the quotient topology with respect to $\pi$.
\end{itemize}

Given $\delta>0$ and $T>1$, a $(\delta,T)$-\emph{chain} from $[x]$ to $[y]$ in $X$ is a finite sequence $[x]=[x_0], [x_1],\cdots,[x_l]=[y]$ of points in $X$ and numbers $\tau_1,\cdots,\tau_l\in \mathbb{R}$ such that $\tau_j\geq T$ and $\overline{d}(\overline{\Phi}_{\tau_j}([x_{j-1}]),[x_j])<\delta$ for all $1\leq j\leq l$. As $\pi:M\to X$ is a submetry and $\Phi$ covers $\overline{\Phi}$, this amounts to asking that the sequence $x=x_0, x_1,\cdots,x_l=y$ of points in $M$ satisfy $d(\Phi_{\tau_j}(x_{j-1}),x_j)<\delta$ for all $1\leq j\leq l$. The \emph{chain-recurrent set} of the flow $\overline{\Phi}$, denoted by $\overline{R}(\Phi)$, is defined to be the set of points $[x]\in X$ such that for any $\delta>0$ and $T>1$ there exists a $(\delta,T)$-chain starting and ending at $[x]$. It follows that the chain-recurrent set is closed and invariant under the flow $\overline{\Phi}$, so that $R(\Phi)=\pi^{-1}(\overline{R}(\Phi))$ is saturated in $M$ as well as closed and invariant under the flow $\Phi$. Note that $R(\Phi)$ can be thought of as the chain-recurrent set for the flow $\Phi$ over $(M,d)$.

A \emph{covering space} $E$ over $G\rr M$ is a covering space $p:E\to M$ equipped with a right $G$-action\footnote{A \emph{right action} of a Lie groupoid $G\rr M$ along a smooth map $p:E\to M$ is a smooth map $\cdot:E\times_MG\to E$ such that $p(e\cdot g)=s(g)$, $(e\cdot g)\cdot g'=e\cdot (gg')$ for all $(g,g')\in G^{(2)}$, and $e\cdot 1_x= e$ for all $x\in M$.} $E\times_{M}G\to E$ along $p$. Note that the map $p$ extends to a Lie groupoid morphism $p:E\rtimes G\to G$ from the action groupoid $E\rtimes G$ into $G$ defined by $p(e,g)=g$, which covers the covering map $E\to M$. Besides, the action groupoid $E\rtimes G\rr E$ is a proper \'etale Lie groupoid as well, meaning that it also represents an orbifold \cite{Moerd}. By \cite[Prop. 2.4]{V}, for any cohomology class $\xi\in H^1(X)$ one can consider a covering space $p_\xi:M_{\xi}\to M$ over $G$ which corresponds to the kernel of the $G$-homomorphism of periods $\textnormal{Per}_\xi$. Such a covering has the property that a $G$-loop $\sigma$ in $M$ lifts to a $(M_\xi\rtimes G)$-loop in $M_\xi$ if and only if the value of the cohomology class $\xi$ on the homology class $[\sigma]\in H_1(X,\mathbb{Z})$ vanishes, so that $\langle \xi,[\sigma]\rangle=\int_{[\sigma]}\xi=0$. We denote by $X_\xi$ the orbifold presented by $M_\xi\rtimes G\rr M_\xi$, thus obtaining an induced map $\overline{p}_\xi:X_\xi\to X$.

Let us now consider the Riemannian metric on $X_\xi$ which makes of $\overline{p}_\xi:X_\xi\to X$ an orbifold Riemannian covering map. Such a Riemannian metric induces a groupoid metric on the proper étale Lie groupoid $M_\xi\rtimes G\rr M_\xi$, so that $p_\xi:M_{\xi}\to M$ becomes a Riemannian covering map. The flow $\overline{\Phi}$ on $X$ lifts uniquely to a flow $\tilde{\Phi}$ on $X_\xi$. Thus, we can consider the chain-recurrent set  $\overline{R}(\tilde{\Phi})\subset X_\xi$ of the lifted flow and denote by $\overline{R}_\xi=\overline{p}_\xi(\overline{R}(\tilde{\Phi}))$ its projection onto $X$. As consequence of \cite[Prop. 3]{FKLZ2}, it follows that $\overline{R}_\xi$ is a closed and $\overline{\Phi}$-invariant subset of $\overline{R}(\Phi)$, which we refer to as the \emph{chain-recurrent set associated with the cohomology class} $\xi$. Its complement $\overline{R}(\Phi)-\overline{R}_\xi$ is denoted by $\overline{C}_\xi$. As before, we also denote by $C_\xi=\pi^{-1}(\overline{C}_\xi)$ and $R_\xi=\pi^{-1}(\overline{R}_\xi)$, the closed and saturated subsets in $M$ which enjoy of similar properties once we lift the flow $\Phi$ to $M_\xi$ using $p_\xi$.

In order to state our main result, we need to consider the homology class associated with any $(\delta,T)$-cycle, an ingredient which was initially exploited in \cite{FKLZ1,FKLZ2}. A  $(\delta,T)$-\emph{cycle} for the flow $\Phi$ is defined as a pair $(x,\tau)$, where $x\in M$ and $\tau > T$ such that $d(\Phi_\tau(x),x)< \delta$. In other words, a $(\delta,T)$-cycle is just a $(\delta,T)$-chain for $l=1$. Once again, using the flow $\overline{\Phi}$, this is equivalent to having a pair $([x],\tau)$, where $[x]\in X$ and $\tau > T$ such that $\overline{d}(\overline{\Phi}_\tau([x]),[x])< \delta$, as $\pi:M\to X$ is a submetry. By assuming that $\delta$ is small enough, one can construct a unique homology class $z\in H_1(X,\mathbb{Z})$ in a canonical way as follows. A pair of positive real numbers $(\epsilon,\delta)$, both of them depending on a non-zero cohomology class $\xi\in H^1(X)$, is said to be a \emph{scale} for $\xi$ if $\xi|_B=0$ for any $\overline{d}$-ball $B$ in $X$ of radius $2\epsilon$ and any two points $[x],[y]\in X$ with $\overline{d}([x],[y])<\delta$ can be connected through a $G$-path joining $x,y\in M$ whose image under $\pi$ is contained within a $\overline{d}$-ball in $X$ of radius $\epsilon$. 

\begin{lemma}
Such a scale always exists.
\end{lemma}
\begin{proof}
Indeed, a scale can be constructed by using Lemma 2 from \cite{FKLZ2} together the groupoid version of the Poincaré Lemma, which was proven in Lemma 8.5 from \cite{PPT}.
\end{proof}

Let us take a $(\delta,T)$-cycle with $\delta$ small enough so that the pair $(\delta,\epsilon)$ determines a scale for $\xi$ and choose a $G$-path $\sigma_x$, joining $\Phi_\tau(x)$ with $x$, such that its image under $\pi$ lies within a $\overline{d}$-ball in $X$ of radius $\epsilon$. This way, we obtain a singular cycle $z$ in $H_1(X,\mathbb{Z})$, which is a combination of the parts of the flow trajectory from $x$ to $\Phi_{\tau}(x)$ and the $G$-path $\sigma_x$, compare \cite[p. 23]{Be} or \cite[Prop. 2.2]{V}. Similarly to \cite[p. 1459]{FKLZ2}, it is simple to promote the latter construction to get a singular cycle in $H_1(X,\mathbb{Z})$ for a $(\delta,T)$-chain, provided that we continue assuming that $\delta$ small enough.

\section{Proof of Theorem \ref{Main Thm}}\label{S:4}

As alluded to previously, the main result of this paper provides topological conditions which guarantee that for a given vector field $\overline{v}$ on $X$ there exists a Lyapunov 1-form $\overline{\omega}$ lying in a prescribed cohomology class $\xi\in H^1(X)$. In the special case $\xi=0$, we get that $\overline{C}_\xi$ is empty and $\overline{R}=\overline{R}(\Phi)$ equals $\overline{R}_\xi$. Consequently, when $\xi=0$, our main result reduces to an orbifold version of a celebrated theorem of Conley \cite{Conley1,Conley2}, which reads as follows.

\begin{lemma}\label{OrbifoldConley}
Let $\overline{v}$ be a vector field on a compact orbifold $X$ and let $\overline{\Phi}:\mathbb{R}\times X\to X$ be the flow generated by $\overline{v}$, whose chain recurrent set is denoted by $\overline{R}$. Then, there exists a smooth Lyapunov function $\overline{L}:X\to \mathbb{R}$ for $(\overline{\Phi},\overline{R})$, meaning that $\overline{v}(\overline{L})<0$ on $X-\overline{R}$ and $d\overline{L}=0$ pointwise on $\overline{R}$.
\end{lemma}
\begin{proof}
Using the classical Conley's theorem, (see Proposition 2 in \cite{FKLZ1} or Corollary 2.3 in \cite{FathiPageault}), one may ensure that there exists a smooth Lyapunov function $L:M\to \mathbb{R}$ for $(\Phi,R)$, where $\Phi:M\times \mathbb{R}\to M$ is the flow generated by $v_0$ and whose chain recurrent set is $R=\pi^{-1}(\overline{R})$. In other words, we have that $v_0(L)<0$ on $M-R$ and $dL=0$ pointwise on $R$. If $\lbrace \mu^x\rbrace_{x\in M}$ is a proper Haar measure system on $G$ then we define
$$L^\mu(x)=\int_{s^{-1}(x)}(L\circ t)(g)d\mu^{x}(g),\qquad x\in M.$$

This is a smooth basic function on $M$, preserving the properties of $L$ mentioned above, so that its induced smooth function $\overline{L}^\mu:X\to \mathbb{R}$ has the properties we desire.
\end{proof}

We are now in conditions to prove our main result.

\begin{proof}[Proof of Theorem \ref{Main Thm}]
Recall that $(X,\overline{d})$ is a compact, locally path-connected metric space. Thus, the main idea of the proof is to use the machinery developed in the previous sections to adapt, to the orbifold setting, the key arguments from the proof of Theorem 1 in \cite{FKLZ1}, which in turn are based on the case of continuous flows on compact, locally path-connected metric spaces, studied in \cite{FKLZ2}.

(a) implies (b). Note that $\overline{C}_\xi$ is compact. Therefore, this implication follows exactly as in the proof of Proposition 4 in \cite{FKLZ2}.

(b) implies (c). We say that $[x]\in X$ is a \emph{quasi-regular point} of the flow $\overline{\Phi}:\mathbb{R}\times X\to X$ if for any continuous function $\overline{f}:X\to \mathbb{R}$, the following limits exist:
$$\lim_{\tau\mapsto \infty} \frac{1}{\tau}\int_0^\tau \overline{f}(\overline{\Phi}_\zeta([x])) d\zeta=\lim_{\tau\mapsto \infty}  \frac{1}{\tau}\int_0^\tau c(\Phi_\zeta(x))f(\Phi_\zeta(x)) d\zeta,$$
where $f:M\to \mathbb{R}$ stands for the continuous basic function presenting $\overline{f}$ and $c:M\to \mathbb{R}$ is a cut-off function. We have that:
\begin{itemize}
\item[$(\star)$] the subset of all quasi-regular points has full measure with respect to any $\overline{\Phi}$-invariant positive measure on $X$, see \cite[p. 106]{Jacobs},
\item[$(\star\star)$] for any quasi-regular point $[x]\in X$, there exists a unique $\overline{\Phi}$-invariant positive measure $\overline{\mu}_{x}$ on $X$ with $\overline{\mu}_{x}(X)=1$ satisfying
$$\lim_{\tau\mapsto \infty} \frac{1}{\tau}\int_0^\tau \overline{f}(\overline{\Phi}_\zeta([x])) d\zeta=\int_X \overline{f}([y])d\overline{\mu}_{x}([y]),$$
for any continuous function $\overline{f}:X\to \mathbb{R}$, compare Riesz representation theorem \cite[p. 256]{Rudin}, and
\item[$(\star\star\star)$] any $\overline{\Phi}$-invariant positive measure $\overline{\mu}$ on $X$ with $\overline{\mu}(X)=1$ sits inside the weak* closure of the convex hull of the set of measures $\lbrace \overline{\mu}_{x}: [x]\in X\ \textnormal{quasi-regular point}\rbrace$, see \cite[p. 108]{Jacobs}.
\end{itemize}

From Theorem 2 in \cite{FKLZ2}, we know that $\overline{C}_\xi$ is closed, and hence compact. Therefore, we can apply the facts mentioned above to restrict the flow $\overline{\Phi}$ to $\overline{C}_\xi$. In particular, we can restrict the flow  $\Phi$ to $C_\xi$. Let $\omega$ be a closed basic 1-form representing $\xi$. If $[x]\in \overline{C}_\xi$ is a quasi-regular point of the flow $\overline{\Phi}|_{\overline{C}_\xi}$ then
\begin{eqnarray*}
\lim_{\tau\mapsto \infty} \frac{1}{\tau}\int_{\overline{\gamma}_x}\overline{\omega} & = & \lim_{\tau\mapsto \infty} \frac{1}{\tau}\int_0^\tau\iota_{\overline{v}}\overline{\omega}(\overline{\Phi}_\zeta([x]) d\zeta =\lim_{\tau\mapsto \infty} \frac{1}{\tau}\int_0^\tau c(\Phi_\zeta(x))\iota_{v_0}\omega(\Phi_\zeta(x)) d\zeta\\
& = & \int_M c(x)\iota_{v_0}\omega(x) d\mu_{x}(y)=\int_X \iota_{\overline{v}}\overline{\omega}([x]) d\overline{\mu}_{x}([y])=\langle \xi, \mathcal{A}_{\mu_x}\rangle,
\end{eqnarray*}
where $\overline{\gamma}_x=\pi\circ \gamma_x$ with $\gamma_x$ denoting the flow trajectory from $x$ to $\Phi_\tau(x)$, $c:M\to \mathbb{R}$ is a cut-off function, and $\mu_{x}$ denotes the $G$-invariant measure on $M$ which is uniquely determined by $\overline{\mu}_x=\mu_x\circ \pi_!$. As a consequence, condition (c) is equivalent to the following statement: the subset $\overline{C}_\xi$ is closed and there exists a constant $\lambda>0$ such that for any quasi-regular point $[x]\in \overline{C}_\xi$ it holds that
\begin{equation}\label{Aux Condition 1}
\lim_{\tau\mapsto \infty} \frac{1}{\tau}\int_{\overline{\gamma}_x}\overline{\omega} \leq -\lambda,
\end{equation}
where $\omega$ is an arbitrary closed basic 1-form on $M$ lying in the cohomology class $\xi\in H^1(X)$. Arguing exactly as we did to define the asymptotic cycle of $\xi$ one can guarantee that the limit \eqref{Aux Condition 1} does not depend on the choice of the closed basic 1-form $\omega$ on $M$. 

Applying Lemma 6 in \cite{FKLZ2} we know that there exist constants $\alpha>0$ and $\beta>0$ such that for any $[x]\in \overline{C}_\xi$ and $\tau\geq 0$ it holds that
$$\int_{\overline{\gamma}_x}\overline{\omega} \leq-\alpha\tau+\beta.$$

This implies that for $\tau\geq 2\beta/\alpha$ one obtains $\frac{1}{\tau}\int_{\overline{\gamma}_x}\overline{\omega} \leq-\frac{\alpha}{2}$, meaning that $\langle \xi, \mathcal{A}_{\mu_x}\rangle \leq-\frac{\alpha}{2}<0$. Therefore, as consequence of ($\star$) and ($\star\star\star$), it follows that for any $\overline{\Phi}$-invariant positive measure $\overline{\mu}$ on $X$ with $\overline{\mu}(X)=\overline{\mu}(\overline{C}_\xi)=1$ we get:

$$
	\langle \xi, \mathcal{A}_{\mu}\rangle \leq-\frac{\alpha}{2}<0.
$$

The rest of the argument follows exactly as in \cite[p. 8]{FKLZ1}, after using \cite[Prop. 4.1.17]{KatokHasselblatt} together with the groupoid version of the Poincaré Lemma \cite[Lem. 8.5]{PPT}. 

Since (c) clearly implies (d), it remains only to show that (d) implies (a). As in the manifold case, our argument here relies on adapting some of Schwartzman’s techniques developed in \cite{Sch} to the orbifold setting. First of all, let us show that there exists a closed basic 1-form $\omega_1$ in the cohomology class $\xi$ such that $\iota_{\overline{v}}\overline{\omega}_1<0$ on $\overline{C}_\xi$. We denote by $\mathcal{D}\subset C(X)=C(M)^G$ the space of functions 
$$\mathcal{D}=\lbrace \overline{v}(\overline{f}): \overline{f}\in C^\infty(X)\rbrace\cong \lbrace v_0(f): f\in C^\infty(M)^G\rbrace,$$
and by $\mathcal{C}^-\subset C(X)=C(M)^G$ the cone of functions
$$\mathcal{C}^-=\lbrace \overline{f}: \overline{f}([x])<0\ \textnormal{for all}\ [x]\in \overline{C}_\xi \rbrace\cong \lbrace f: f(x)<0\ \textnormal{for all}\ x\in C_\xi \rbrace.$$

The fact that $\overline{C}_\xi$ is compact implies that the cone $\mathcal{C}^-$ is open in the Banach space $C(X)$ equipped with the standard uniform (supremum) norm. Let us now pick an arbitrary closed basic 1-form $\omega$ presenting the cohomology class $\xi$. We claim that $\mathcal{C}^-\cap (\iota_{\overline{v}}\overline{\omega}+\mathcal{D})\neq \emptyset$. By contradiction, we suppose that $\mathcal{C}^-\cap (\iota_{\overline{v}}\overline{\omega}+\mathcal{D})= \emptyset$, so that there exists a continuous linear functional $\Xi:C(X)\to \mathbb{R}$ verifying 
$$\Xi|_{(\iota_{\overline{v}}\overline{\omega}+\mathcal{D}) }\geq 0\qquad \textnormal{and}\qquad \Xi|_{\mathcal{C}^-}<0.$$

The existence of $\Xi$ is consequence of the Hahn--Banach theorem, see \cite[p. 58]{Rudin}. Observe that the restriction of $\Xi$ to $\mathcal{D}$ vanishes, as $\iota_{\overline{v}}\overline{\omega}+\mathcal{D}$ is an affine subspace and $\Xi$ is bounded over it from below. Additionally, the Riesz representation theorem ensures that there exists a measure $\overline{\mu}$ on $X$ for which $\Xi(\overline{f})=\int_X\overline{f}d\overline{\mu}([x])$ for all $f\in C(X)$. Such a measure is $\overline{\Phi}$-invariant since $\Xi|_{\mathcal{D}}=0$ (compare \cite{Sch}), and, moreover, from the very definition of $\mathcal{C}^-$ it holds that $\overline{\mu}|_{\overline{C}_\xi}>0$ since $\Xi|_{\mathcal{C}^-}<0$.

Let $\overline{\chi}:X\to \mathbb{R}$ denote the characteristic function of $\overline{C}_\xi$. It follows that the measure $\overline{\nu}=\overline{\chi}\cdot \overline{\mu}$ is $\overline{\Phi}$-invariant and positive, as $\overline{C}_\xi$ is $\overline{\Phi}$-invariant. Furthermore, the fact that any $\overline{\Phi}$-invariant measure is supported on $\overline{R}(\Phi)=\overline{R}_\xi\cup \overline{C}_\xi$ implies that the $\overline{\Phi}$-invariant measure $\overline{\mu}-\overline{\nu}$ is supported on $\overline{R}_\xi$. Therefore, using our assumption $\xi|_{\overline{R}_\xi}=0$ together with similar arguments as those used to obtain \cite[Eq. (7.5)]{FKLZ1} one gets that $\langle \xi, \mathcal{A}_{\mu-\nu}\rangle=0$. Consequently, 
$$\langle \xi, \mathcal{A}_{\nu}\rangle=\langle \xi, \mathcal{A}_{\mu}\rangle=\int_X\iota_{\overline{v}}\overline{\omega}([x])d\overline{\mu}([x])=\Xi(\iota_{\overline{v}}\overline{\omega})\geq 0,$$
which contradicts (d). In other words, $\mathcal{C}^-\cap (\iota_{\overline{v}}\overline{\omega}+\mathcal{D})\neq \emptyset$ and therefore there exists a smooth basic function $h:M\to \mathbb{R}$ such that the closed basic 1-form $\omega_1=\omega+dh$ also represents the cohomology class $\xi$ and verifies $\iota_{\overline{v}}\overline{\omega}_1<0$ on $\overline{C}_\xi$.

Since $\overline{C}_\xi\subset X$ is compact we can consider an open neighborhood $W_1$ of $\overline{C}_\xi$ such that $W_1\cap \overline{R}_\xi=\emptyset$ and $\iota_{\overline{v}}\overline{\omega}_1<0$ on $W_1$. Also, the fact that $\xi|_{\overline{R}_\xi}=0$ implies that there is another open neighborhood $W_2$ of $\overline{R}_\xi$ such that $W_1\cap W_2=\emptyset$, and, using $G$-invariant partitions of unity as in \cite[s. 3.1]{CM0}, there exists a smooth basic function $h_1:M\to \mathbb{R}$ verifying $\omega_1|_{\pi^{-1}(W_2)}=dh_1$ and $dh_1|_{\pi^{-1}(W_1)}=0$. We can use now Lemma \ref{OrbifoldConley} to obtain a Lyapunov function $\overline{L}:X\to \mathbb{R}$ for $(\overline{\Phi},\overline{R}(\Phi))$, which we assume to be presented by a smooth basic function $L:M\to \mathbb{R}$. Let us define the closed basic 1-form on $M$:
\begin{equation}\label{Aux Condition 2}
\omega_2=\omega_1-dh_1+a\cdot dL,
\end{equation}
where $a>0$ is a constant that we still need to choose appropriately. It is clear that $\omega_2$ represents $\xi$ as well. By construction, $\omega_2|_{\pi^{-1}(W_2)}=a\cdot dL$ for any $a>0$, so that it satisfies property (2) of Definition \ref{MainDefinition}. Moreover, we also have that $\iota_{\overline{v}}\overline{\omega}_2<0$ on both $W_1$ and $W_2-\overline{R}_\xi$, because $\omega_1-dh_1$ restricted to either $\pi^{-1}(W_1)$ or $\pi^{-1}(W_2-\overline{R}_\xi)$ vanishes, whereas $\overline{v}(\overline{L})<0$. Finally, note that the complement of $W_1\cup W_2$ is compact and disjoint from $\overline{R}(\Phi)$. Hence,
$$1<a_0=1+\sup_{[x]\notin W_1\cup W_2}\frac{\vert \iota_{\overline{v}}(\overline{\omega}_1-\overline{dh_1}) \vert }{\vert \overline{v}(\overline{L})\vert }<\infty, $$
thus obtaining that $\iota_{\overline{v}}\overline{\omega}_2<0$ on $X- \overline{R}_\xi$ for all $a\geq a_0$. The latter implies that for all choices $a\geq a_0$ it holds that $\overline{\omega}_2$ also satisfies property (1) of Definition \ref{MainDefinition} for $(\overline{\Phi},\overline{R}_\xi)$. This completes the proof.
\end{proof}

\section{Examples}\label{S:5}

Motivated by the constructions developed in \cite[s. 7]{FKLZ2}, in this short section we provide some examples which allow us to illustrate our main result. 

Let us first suppose that $K$ is a discrete group acting properly and effectively on a smooth manifold $M$, so that the action groupoid $K\ltimes M\rr M$ represents an orbifold $X\approx M/K$. Differential forms on $M/K$ are completely determined by differential forms on $M$ which are $K$-invariant and horizontal, the latter term meaning that they vanish on vectors tangent to the $K$-orbits. As shown in \cite{Wa}, a differential form is basic and horizontal in the previous sense if and only if it is a basic differential form on the corresponding action groupoid. Similarly, vector fields on $M/K$ are completely determined by $K$-invariant vector fields on $M$ and any flow $\overline{\Phi}:\mathbb{R}\times M/K\to M/K$ is covered by the corresponding $K$-equivariant flow $\Phi:\mathbb{R}\times M\to M$.

\begin{example}[Pillowcase]
	Consider the 2-torus $\mathbb{T}^2=\mathbb{R}^2/\mathbb{Z}^2$ embedded in $\mathbb{R}^3$ together with the action by the group $\mathbb{Z}_2$ given by rotations of angle $\pi$ around one of the axis of $\mathbb{T}^2$. The quotient space $\mathbb{T}^2/\mathbb{Z}_2$ is an orbifold whose underlying space is homeomorphic to the 2-sphere and whose singular locus consists of four singular points each with isotropy $\mathbb{Z}_2$.
\end{example}

\begin{example}
Let $Q$ be the pillowcase. Take a cohomology class $\xi\in H^1(Q)$ which can be written as $\xi=f_1[dx]+f_2[dy]$ where $f_1,f_2\in C^{\infty}(\mathbb{T}^2)$ are smooth odd functions, and $dx$, $dy$ stand for the standard coordinate 1-forms. We consider the flow $\Phi$ of the following $\mathbb{Z}_2$-invariant vector field
$$v_0=h(x,y)\cdot \left( a\frac{\partial}{\partial x}+b\frac{\partial}{\partial y} \right),$$ 
where $b\neq 0$, $a/b\in \mathbb{Q}$, and $h:\mathbb{T}^2\to [-1,1]$ is a smooth odd function, vanishing at a single point $(x_0,y_0)\in \mathbb{T}^2$. The chain recurrent set $\overline{R}(\Phi)$ is the whole pillowcase $Q$, while $\overline{R}_\xi$ equals $Q$ if $f_1a+f_2b=0$, or $h^{-1}(0)=\mathbb{Z}_2\cdot (x_0,y_0)$ otherwise. If $f_1a+f_2b\neq 0$ then the set $\overline{C}_\xi=Q-[(x_0,y_0)]$ is not closed. However, we have that a Lyapunov 1-form in $\xi$ exists if and only if $f_1a+f_2b<0$, thus obtaining that $\overline{\omega}=f_1dx+f_2dy$ is such a Lyapunov 1-form.

Let us now suppose that $v_0$ is such that $b\neq 0$, $a/b\notin \mathbb{Q}$, and $h:\mathbb{T}^2\to \mathbb{R}$ is a smooth odd function which has no zeros. If $\xi$ is such that $f_1a+f_2b= 0$ then $\overline{R}(\Phi)=Q$ and $\overline{R}_\xi=\emptyset$, but condition (b) in Theorem \ref{Main Thm} is not satisfied. Hence, there is no Lyapunov 1-form for $(\overline{\Phi},\overline{R}_\xi)$ in $\xi$.
\end{example}

The following general procedure describes a class of examples of flows $\overline{\Phi}:\mathbb{R}\times X\to X$ for which there exists a cohomology class $\xi\in H^1(X)$ satisfying the conditions of our main theorem.

\begin{example}
Let $X'$ be a compact orbifold endowed with a fixed vector field $\overline{v}'$, whose flow is denoted by $\overline{\Psi}$. We assume that the chain recurrent set $\overline{R}(\Psi)$ equals $\overline{R}_1\cup \overline{R}_2$ where $\overline{R}_1$ and $\overline{R}_2$ are closed and disjoint. Adapting Example 1 in \cite{FKLZ2} to the orbifold context, one can build a flow $\overline{\Phi}$ on the product orbifold $X=X'\times S^1$ such that $\overline{R}_\xi(\Phi)=\overline{R}_1\times S^0$ and $\overline{C}_\xi(\Phi)=\overline{R}_2\times S^1$. Here $\xi\in H^1(X)$ stands for the cohomology class induced by the projection onto the circle $X\to S^1$ and $S^0$ is the two-point set. 

Let $\theta\in [0,2\pi]$ denote the angle coordinate on $S^1$, and consider the vector fields on $S^1$ given by $w_1=\cos(\theta)\cdot \frac{\partial}{\partial \theta}$ and $w_2=\frac{\partial}{\partial \theta}$. Observe that $w_1$ has two zeros $\lbrace p_1,p_2\rbrace=S^0\subset S^1$ corresponding to the angles $\theta=\pi/2$ and $\theta=3\pi/2$. Using orbifold (i.e. $G$-invariant) partitions of unity as in \cite[s. 3.1]{CM0}, there exist smooth functions $\overline{f}_1,\overline{f}_2:X'\to [0,1]$ having disjoint supports and satisfying $\overline{f}_1|_{\overline{R}_1}=1$  and $\overline{f}_2|_{\overline{R}_2}=1$. Hence, the flow $\overline{\Phi}:\mathbb{R}\times X\to X$ of the vector field $\overline{v}$ on $X$ defined by
$$\overline{v}=\overline{v}'+ \overline{f}_1w_1+\overline{f}_2w_2,$$
is such that $\overline{R}_\xi(\Phi)=\overline{R}_1\times S^0$ and $\overline{C}_\xi(\Phi)=\overline{R}_2\times S^1$. This immediately implies that $\overline{C}_\xi(\Phi)$ is closed and that $\xi|_{\overline{R}_\xi(\Phi)}=0$. Moreover, it is simple to check that condition (b) in Theorem \ref{Main Thm} is also satisfied.
\end{example}

\end{document}